\documentclass[10pt,oneside,a4paper]{article} 			

\usepackage[fleqn]{amsmath} 							
\usepackage{amsthm,amsfonts,latexsym,amssymb,amscd} 	
\usepackage[mathscr]{eucal}   							
\usepackage[all,2cell]{xy} \UseAllTwocells 				
\usepackage{framed}										
\usepackage{color}										
\usepackage{txfonts}									

\definecolor{cmyk}{cmyk}{.0,.4,.9,.5}

\usepackage[draft=false]{hyperref} 						

\DeclareFontFamily{U}{rsfs}{\skewchar\font127 }
\DeclareFontShape{U}{rsfs}{m}{n}{%
   <5> <6> rsfs5
   <7> rsfs7
   <8> <9> <10> <10.95> <12> <14.4> <17.28> <20.74> <24.88> rsfs10
}{}
\DeclareSymbolFont{rsfs}{U}{rsfs}{m}{n} 
\DeclareSymbolFontAlphabet{\scr}{rsfs}					

\DeclareMathOperator{\id}{Id}

\DeclareMathOperator{\Hom}{Hom}

\renewcommand{\emph}{\textbf} 							

\newcommand{\hlink}[2]{\href{#1}{\texttt{#2}}} 			

\newcommand{\xqedhere}[2]{%
  \rlap{\hbox to#1{\hfil\llap{\ensuremath{#2}}}}}

\theoremstyle{plain}
\newtheorem{theorem}{Theorem}[section]					

\newtheorem{proposition}[theorem]{Proposition}

\newtheorem{definition}[theorem]{Definition}

\theoremstyle{definition} 

\newtheorem{example}[theorem]{Example}  

\numberwithin{equation}{section}  						
\title{\textbf{Involutive Weak Globular Higher Categories}\footnote{
This is a reformatted and corrected version, only for arXiv purposes, of a paper presented by the first author (on 03 June 2017) and already published in the Proceedings of the 22$^{\rm nd}$ Annual Meeting in Mathematics (AMM 2017), 2-4 June 2017, Department of Mathematics, Faculty of Science, 
Chiang Mai University, Chiang Mai, Thailand.}}

\author{\normalsize  
Paratat Bejrakarbum $^a$
Paolo Bertozzini $^b$  
\\  
\normalsize \textit{Department of Mathematics and Statistics, Faculty of Science and Technology,}
\\
\normalsize \textit{Thammasat University, Pathumthani 12121, Thailand} 
\\ 
\normalsize $^a$ e-mail: \texttt{paratat\_p35@hotmail.com}
\\
\normalsize $^b$ e-mail: \texttt{paolo.th@gmail.com}
}

\date{\normalsize{submitted: 18 April 2017 \quad presented/published: 03 June 2017 \quad arXiv version: 02 August 2017}}

\begin{document}

\maketitle 

\begin{abstract} \noindent 
We investigate the notion of involutive weak globular $\omega$-categories
via Jacque Penon's approach. In particular, we give the constructions of a free self-dual globular $\omega$-magma, of a free strict involutive globular 
$\omega$-category, over an $\omega$-globular set, and a contraction between them.
The monadic definition of involutive weak globular $\omega$-categories is given as usual via algebras for the monad induced by a certain adjunction.
In our case, the adjunction is obtained from the ``free functor'' that associates to every $\omega$-globular set the above contraction.
Some examples of involutive weak globular $\omega$-categories are also provided.

\medskip


\smallskip

\noindent
\emph{MSC-2010:} 					
					18D05,			
					18D99, 
					46M99. 			


\end{abstract}

\tableofcontents


\section{Introduction}\label{sec: intro}

Higher category theory is a subject that is currently receiving a lot of interest, with strong links not only with algebraic topology (where we can trace its origins), but with logic, computer science, foundations of mathematics, mathematical physics, general systems' theory and more 
(see~\cite{BS,L2,H} and for some speculative applications to relational quantum theory also~\cite{B}). 

Although higher category theory, was somehow implicitly present at the time of the very inception of the subject in the work of 
S.Eilenberg-S.Mac Lane~\cite{EM} (natural transformations are just an example of globular 2-arrows in a strict 2-category), 
strict \hbox{$n$-categories} where originally defined by C.Ehresmann, both in their cubical~\cite{E63} and globular versions~\cite{E65} and 
M.Kelly-S.Eilenberg's enriched categories~\cite{EK} allow an iterative construction of strict higher categories.  

Weak categories (categories where the algebraic axioms of associativity and unitality hold only up to higher-level isomorphism) formally appear in the definition of weak monoidal categories~\cite{Be1,M} (a monoidal category is essentially a 2-category with only one object) and in their 
``many-objects'' (horizontally categorified) version as J.B\'enabou's  bicategories~\cite{Be}. 

J.Roberts, the pioneer of application of category theory to the study of algebraic quantum field theory in physics, was apparently the first to consider strict globular $\omega$-categories (categories equipped with an infinite tower of higher morphisms and compositions)~\cite{R}. 
Strict cubical $\omega$-groupoids and categories appeared almost at the same time in a series of works by R.Brown-P.Higgins~\cite{BH}, motivated by their generalization of Seifert-Van Kampen theorem in algebraic topology (see the nice recent textbook~\cite{BHS} for details). 

A.Grothendieck~\cite{G} in his famous ``Pursuing Stacks'' manuscript (partially inspired by discussions with R.Brown and collaborators in Bangor) described strict globular $\omega$-categories and proposed the use of weak-$\omega$-groupoids as a way to capture the homotopy content of spaces.  
The actual definition of weak $n$-categories (for $n>2$ or $n=\omega$), starting with R.Street's definition of weak $\omega$-category based on the algebra of ``symplexes''~\cite{St}, has been a quite laborius (and still ongoing) process with several alternative partially equivalent definitions under discussion.\footnote{For a general background comparison we refer to the excellent introductions by T.Leinster~\cite{L2}, 
E.Cheng-A.Lauda~\cite{CL} and, for a quite useful  historical account of the complicated developments, to the bibliographical appendix contained in T.Leinster~\cite{L1}.} 

Algebraic definitions of weak gobular $\omega$-categories, based on suitable monads, have been developed by M.Batanin~\cite{Ba1,Ba2}, 
J.Penon~\cite{P}, T.Leinster~\cite{L2} and later C.Kachour~\cite{K} (see also the alternative view of G.Kondratiev~\cite{Ko}).  

The notion of strict involution in category theory (an involutive endo-functor) was apparently repeatedly rediscovered and utilized in different contexts, usually with additional structures in place, before being recently formalized through ``dagger categories'' by P.Selinger~\cite{S}. Strict involutions appear in the ``categories with involution'' (M.Burgin~\cite{Bu}, J.Lambek~\cite{La}, \dots) with a compatible ``order relation''; in ``allegories''  
(P.Freyd-A.Scedrov~\cite{FS}) where a further operation of ``intersection'' appears; in the definitions of ``$*$-category'' and ``$*$-algebroid'' in the literature on \hbox{C*-categories} starting with P.Ghez-R.Lima-J.Roberts~\cite{GLR} and P.Mitchener~\cite{Mi}, where involutions are supposed to be conjugate-linear on the $\Hom$-spaces; in the works on ``compact closed categories'' starting with S.Abramsky-B.Coecke~\cite{AC}. 

Involutions for strict globular $n$-categories (as involutive endofuctors that are covariant or contravariant for the several compositions) have been studied in~\cite{BCLS, B} (see also~\cite{BCL1,BCL2}) and for the case of strict double categories (strict cubical $2$-categories) in~\cite{BCM}. 
The study of weak forms of involutions in (higher) category theory had a more intricate evolution (that we will not investigate here) strongly linked with the study of ``dualizing objects'' and $*$-autonomous categories~\cite{BaW}. 

A notion of involutive weak monoidal category is contained in~\cite{BCL3} and an alternative definition was proposed by J.Egger~\cite{Eg}. 

\textit{
As a very first step towards a possible treatment of weak higher C*-categories, in the present work, our main purpose is to put forward a definition of \emph{involutive weak higher category} in the context of J.Penon's definition of weak globular $\omega$-category~\cite{P}. 
Possible immediate future extensions of this research will examine the notion of involutions for M.Batanin~\cite{Ba1} and T.Leinster~\cite{L2} algebraic approaches to higher categories as well. 
}

Here we proceed to describe in some detail the content of the paper. 

In section~\ref{sec: prelim} we briefly review the basic notions on strict higher categories that we need. 
In order to make immediate contact with the already available works on J.Penon's approach, we decide here to define strict higher categories via  ``higher quivers'', whose  definition is recalled in subsection~\ref{subsec: sghc}. 
Previous work on higher categories~\cite{BCLS,B} utilized an algebraic definition of strict higher categories via ``partial monoids on 
$n$-arrows''; a discussion of the categorical equivalence between the two descriptions has been done elsewhere~\cite{Pu,BP}.  
In this paper, we restrict our attention to the case of globular higher quivers and globular higher categories based on them.\footnote{The treatment of cubical higher categories will be the objective of a further separate investigation, as soon as a full study of strict involutive $n$-tuple categories is available, extending the previous work~\cite{BCM} for double categories.} 
Contrary to the treatment in~\cite{BCLS,B}, where only strict $n$-categories are considered, in subsection~\ref{subsec: sghc} we cover also the general case of strict globular $\omega$-categories. 

The definition of strict involutive \hbox{$n$-category} from~\cite{BCLS,B} is similarly extended to the case of strict involutive globular 
$\omega$-categories in subsection~\ref{subsec: inv cat}. 
We remark that also for our strict (involutive) globular $\omega$-categories it is perfectly possible to substitute the ``usual exchange'' axiom with the relaxed ``non-commutative exchange'' property proposed in~\cite{BCLS,B}.  

In order to fix the notation and to make the paper self-contained, monads and their algebras are defined in section~\ref{subsec: ama}. 
The essential features of J.Penon's construction are recalled in section~\ref{subsec: Penon}.  
We do not necessarily require our globular $\omega$-quivers to be initially reflexive (and this should avoid the already known problems described 
in~\cite{ChM}). 

The main subject of this work is in section~\ref{sec: main}.  An explicit definition and construction of free self-dual globular $\omega$-magmas and of free strict involutive globular $\omega$-categories over a globular $\omega$-quiver is presented in detail in subsections~\ref{subsec: sd quiver} and~\ref{subsec: inv cat} followed by a similar construction of the free ``involutive'' contraction over a globular $\omega$-quiver in \ref{subsec: w inv cat}. 
In subsection~\ref{subsec: w inv cat} we prove that the forgetful functor from the category $\mathscr{Q}_\omega^*$ of contractions (of involutive globular $\omega$-magmas over strict involutive globular $\omega$-categories) to the category of globular $\omega$-quivers admits a left-adjoint and we give the monadic definition of weak involutive globular $\omega$-categories as algebras for such monad. 
Some preliminary examples are presented in subsection~\ref{subsec: ex}.  

\section{Preliminaries}\label{sec: prelim}

We collect here the background definitions and results that are preliminary to our work. The main references are J.Penon~\cite{P}, 
T.Leinster~\cite{L2}, E.Cheng-A.Lauda~\cite{CL}. 

In our treatment here, we carefully separate the algebraic axioms (associativity, unitality, unital functoriality and exchange) from the ``structural requirements'' introduced via higher quivers, that are not a-priori reflexive. Reflexivity is considered as a nullary partial operation in parallel to the partial binary operation of composition. 

\subsection{Strict Globular Higher Categories}\label{subsec: sghc}

An \textbf{$\omega$-quiver} 
$Q^0 \overset{s^0}{\underset{t^0}{\leftleftarrows}} Q^1 \overset{s^1}{\underset{t^1}{\leftleftarrows}} \cdots
\overset{s^{n-2}}{\underset{t^{n-2}}{\leftleftarrows}} Q^{n-1}\overset{s^{n-1}}{\underset{t^{n-1}}{\leftleftarrows}} Q^n \overset{s^n}{\underset{t^n}{\leftleftarrows}} \cdots$ is an infinite family of sets $Q^k$ for $k\in \mathbb{N}_0$ equipped with infinite pairs of source and target maps $s^k,t^k:Q^{k+1}\rightrightarrows Q^k$ for $k\in \mathbb{N}_0$.
Elements of $Q^m$ are called \textbf{$m$-cells} of $Q$ and their ``shape'' is as follows: 
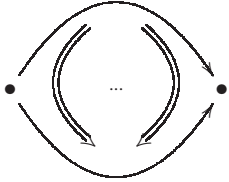
\begin{equation*}
\xy 
0;/r.22pc/:
(0,15)*{};
(0,-15)*{};
(0,8)*{}="A";
(0,-8)*{}="B";
{\ar@{=>}@/_1pc/ "A"+(-4,1) ; "B"+(-3,0)};
{\ar@{=}@/_1pc/ "A"+(-4,1) ; "B"+(-4,1)};
{\ar@{=>}@/^1pc/ "A"+(4,1) ; "B"+(3,0)};
{\ar@{=}@/^1pc/ "A"+(4,1) ; "B"+(4,1)};
{\ar@{}(-5.5,0)*{} ; (5.5,0)*{}|\dots}; 
(-15,0)*+{\bullet}="1";
(15,0)*+{\bullet}="2";
{\ar@/^2.75pc/ "1";"2"};
{\ar@/_2.75pc/ "1";"2"};
\endxy
\end{equation*}

An \textbf{$\omega$-globular set} is an $\omega$-quiver satisfying the globularity condition, i.e. $s^{k-1}s^k=s^{k-1}t^k$ and 
$t^{k-1}s^k=t^{k-1}t^k$ for all $k\in \mathbb{N}$.

An $\omega$-globular set 
$Q^0 \overset{s^0}{\underset{t^0}{\leftleftarrows}} Q^1 \overset{s^1}{\underset{t^1}{\leftleftarrows}} \cdots
\overset{s^{n-1}}{\underset{t^{n-1}}{\leftleftarrows}} Q^n \overset{s^n}{\underset{t^n}{\leftleftarrows}} \cdots$ is \textbf{reflexive} if there exists a family of maps 
$Q^0 \overset{\iota^0}{\rightarrow} Q^1 \overset{\iota^1}{\rightarrow} \cdots \overset{\iota^{n-1}}{\rightarrow}
Q^n \overset{\iota^n}{\rightarrow} \cdots$ such that $s^k\circ \iota^k=\id_{Q^k}=t^k\circ \iota^k$ for every $k\in \mathbb{N}_0$.

A \textbf{(reflexive) globular $\omega$-magma} is  a (reflexive) $\omega$-globular set equipped with a function \\ 
$\circ^m_p:Q^m\times_{Q^p} Q^m \rightarrow Q^m$ for each $0\leq p<m$, where 
$$Q^m\times_{Q^p} Q^m:=\{(x',x)\in Q^m\times Q^m \ | \ t^pt^{p+1}\cdots t^{m-1}(x)=s^ps^{p+1}\cdots s^{m-1}(x')\},$$ such that the following conditions hold: if $0\leq p<m$ and $(x',x)\in Q^m\times_{Q^p} Q^m$,
\begin{itemize}
\item $s^qs^{q+1}\cdots s^{m-1}(x'\circ^m_p x)=\left\{  \begin{array}{ll}
s^qs^{q+1}\cdots s^{m-1}(x')\circ^q_p s^qs^{q+1}\cdots s^{m-1}(x), & q>p; \\
s^qs^{q+1}\cdots s^{m-1}(x'), & q\leq p. \end{array} \right.$
\item $t^qt^{q+1}\cdots t^{m-1}(x'\circ^m_p x)=\left\{  \begin{array}{ll}
t^qt^{q+1}\cdots t^{m-1}(x')\circ^q_p t^qt^{q+1}\cdots t^{m-1}(x), & q>p; \\
t^qt^{q+1}\cdots t^{m-1}(x), & q\leq p. \end{array} \right.$
\end{itemize}
Here is a graphical rendering of $\circ^3_0, \circ^3_1,\circ^3_2$ for $3$-arrows: 
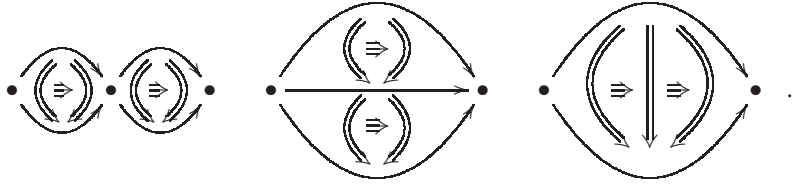
\begin{equation*}
\begin{tabular}{ccc}
\xy 0;/r.22pc/:
(0,10)*{};
(0,-10)*{};
(-14,0)*+{\bullet}="1";
(0,0)*+{\bullet}="2";
{\ar@/^1.33pc/ "1";"2"};
{\ar@/_1.33pc/ "1";"2"};
(14,0)*+{\bullet}="3";
{\ar@/^1.33pc/ "2";"3"};
{\ar@/_1.33pc/ "2";"3"};
(-6.75,4)*+{}="A";
(-6.75,-4)*+{}="B";
{\ar@{=>}@/_.5pc/ "A"+(-1.33,0) ; "B"+(-.66,-.55)};
{\ar@{=}@/_.5pc/ "A"+(-1.33,0) ; "B"+(-1.33,0)};
{\ar@{=>}@/^.5pc/ "A"+(1.33,0) ; "B"+(.66,-.55)};
{\ar@{=}@/^.5pc/ "A"+(1.33,0) ; "B"+(1.33,0)};
(6.75,4)*+{}="A";
(6.75,-4)*+{}="B";
{\ar@{=>}@/_.5pc/ "A"+(-1.33,0) ; "B"+(-.66,-.55)};
{\ar@{=}@/_.5pc/ "A"+(-1.33,0) ; "B"+(-1.33,0)};
{\ar@{=>}@/^.5pc/ "A"+(1.33,0) ; "B"+(.66,-.55)};
{\ar@{=}@/^.5pc/ "A"+(1.33,0) ; "B"+(1.33,0)};
{\ar@3{->} (-8,0)*{}; (-5.5,0)*{}};
{\ar@3{->} (5.5,0)*{}; (8,0)*{}};
\endxy
&
\xy 0;/r.22pc/:
(0,15)*{};
(0,-15)*{};
(0,9)*{}="A";
(0,1)*{}="B";
{\ar@{=>}@/_0.5pc/ "A"+(-2,1) ; "B"+(-1,0)};
{\ar@{=}@/_.5pc/ "A"+(-2,1) ; "B"+(-2,1)};
{\ar@{=>}@/^.5pc/ "A"+(2,1) ; "B"+(1,0)};
{\ar@{=}@/^.5pc/ "A"+(2,1) ; "B"+(2,1)};
{\ar@3{->} (-1.5,6)*{} ; (1.5,6)*{}};
(0,-1)*{}="A";
(0,-9)*{}="B";
{\ar@{=>}@/_.5pc/ "A"+(-2,-1) ; "B"+(-1,-1.5)};
{\ar@{=}@/_.5pc/ "A"+(-2,0) ; "B"+(-2,-.7)};
{\ar@{=>}@/^.5pc/ "A"+(2,-1) ; "B"+(1,-1.5)};
{\ar@{=}@/^.5pc/ "A"+(2,0) ; "B"+(2,-.7)};
{\ar@3{->} (-1.5,-5)*{} ; (1.5,-5)*{}};
(-15,0)*+{\bullet}="1";
(15,0)*+{\bullet}="2";
{\ar@/^2.75pc/ "1";"2"};
{\ar@/_2.75pc/ "1";"2"};
{\ar "1";"2"};
\endxy
&
\xy 0;/r.22pc/:
(0,15)*{};
(0,-15)*{};
(0,8)*{}="A";
(0,-8)*{}="B";
{\ar@{=>} "A"+(0,1.3) ; "B"};
{\ar@{=>}@/_1pc/ "A"+(-4,1) ; "B"+(-3,0)};
{\ar@{=}@/_1pc/ "A"+(-4,1) ; "B"+(-4,1)};
{\ar@{=>}@/^1pc/ "A"+(4,1) ; "B"+(3,0)};
{\ar@{=}@/^1pc/ "A"+(4,1) ; "B"+(4,1)};
{\ar@3{->} (-5.5,0)*{} ; (-2.5,0)*{}};
{\ar@3{->} (2.5,0)*{} ; (5.5,0)*{}};
(-15,0)*+{\bullet}="1";
(15,0)*+{\bullet}="2";
{\ar@/^2.75pc/ "1";"2"};
{\ar@/_2.75pc/ "1";"2"};
\endxy
\end{tabular}.
\end{equation*}

A \textbf{strict globular $\omega$-category} is a reflexive globular $\omega$-magma $\mathscr{C}$ such that:
\begin{enumerate}
\item (associativity) if $0\leq p<m$ and $x,y,z\in \mathscr{C}^m$ with 
\\
$(z,y),(y,x)\in \mathscr{C}^m\times_{\mathscr{C}^p}\mathscr{C}^m$, then
$(z\circ^m_py)\circ^m_px=z\circ^m_p(y\circ^m_px)$,
\item (unitality) if $0\leq p<m$ and $x\in \mathscr{C}^m$, then
$$\iota^{m-1}\cdots \iota^pt^p\cdots t^{m-1}(x)\circ^m_px=x=x\circ^m_p \iota^{m-1}\cdots \iota^ps^p\cdots   s^{m-1}(x),$$
\item (functoriality of identities) if $0\leq q<p$ and $(x',x)\in\mathscr{C}^p \times_{\mathscr{C}^q}\mathscr{C}^p$, then:
$$\iota^p(x')\circ^{p+1}_q \iota^p(x)=\iota^{p}(x'\circ^p_qx),$$ 
\item (binary exchange) if $0\leq q<p<m$ and $x,x',y,y'\in \mathscr{C}^m$ with
\\ 
$(y',y),(x',x)\in \mathscr{C}^m\times_{\mathscr{C}^p}\mathscr{C}^m$ and
$(y',x'),(y,x)\in \mathscr{C}^m\times_{\mathscr{C}^q}\mathscr{C}^m$, then
$$(y'\circ^m_py)\circ^m_q(x'\circ^m_px)=(y'\circ^m_qx')\circ^m_p(y\circ^m_qx), \quad 
\xymatrix{
\bullet \ruppertwocell{x} \rlowertwocell{x'} \ar[r] & \bullet \ruppertwocell{y} \rlowertwocell{y'} \ar[r] 
& \bullet 
}. 
$$
\end{enumerate}

A \textbf{covariant morphism of $\omega$-quivers} $Q,\hat{Q}$ is a family of maps
$\phi^n:Q^n\rightarrow \hat{Q}^n$ such that, for $n\in \mathbb{N}_0$,
$\hat{s}^n \circ \phi^{n+1}=\phi^n \circ s^n$ and $\hat{t}^n \circ \phi^{n+1}=\phi^n \circ t^n$.

A \textbf{covariant morphism of reflexive $\omega$-globular sets} $Q,\hat{Q}$ is a morphism of $\omega$-quivers such that, for 
$n\in \mathbb{N}_0$, $\hat{i}^n\circ \phi^n=\phi^{n+1}\circ i^n$.

A \textbf{covariant morphism of (reflexive) globular $\omega$-magmas} $M,\hat{M}$ is a morphism of (reflexive) \hbox{$\omega$-globular} sets such that 
$\phi^n(x\circ^n_q y)=\phi^n(x)\hat{\circ}^n_q \phi^n(y)$.
Such a morphism is called a \textbf{covariant \hbox{$\omega$-functor}} when $M$ and $\hat{M}$ are strict globular $\omega$-categories.

\subsection{Adjunctions, Monads, Algebras}\label{subsec: ama} 

To make the paper self-contained, we recall some well-known definitions in category theory.\footnote{For background in category theory, among several texts, see~\cite{Bo,BW,ML}.}

Let $\mathscr{C}$ and $\mathscr{D}$ be categories and $F:\mathscr{C} \rightarrow \mathscr{D}$ and $U:\mathscr{D} \rightarrow \mathscr{C}$ be functors.  We say that the functor $F$ is \textbf{left adjoint} to the functor $U$ or the functor $G$ is \textbf{right adjoint} to the functor $F$, denoted by $F\dashv U$ or $U\vdash F$, if there exist natural transformations $\eta:\id_\mathscr{D}\Rightarrow FU$ and 
$\epsilon:UF\Rightarrow \id_\mathscr{C}$ making commutative the following diagrams:
\begin{equation*}
\vcenter{\xymatrix{F \ar[r]^{\eta F} \ar[dr]_{1_F} & FUF \ar[d]^{F\epsilon}
\ar@{}[dl]|(0.35){\circlearrowleft} \\ & F}} \quad 
\vcenter{\xymatrix{U \ar[r]^{U\eta}
\ar[dr]_{1_U} & UFU \ar[d]^{\epsilon U} \ar@{}[dl]|(0.35){\circlearrowleft} \\ & U}} 
\quad \text{that is, $F\epsilon \circ \eta F=1_F$ and $\epsilon U\circ U\eta =1_U$.}
\end{equation*}

A \textbf{monad} $(T,\mu,\eta)$ on a category $\mathscr{C}$ consists of a functor $T:\mathscr{C} \rightarrow \mathscr{C}$ and natural transformations $\eta :1_\mathscr{C} \Rightarrow T$ (the \textbf{unit}) and $\mu :T^2 \Rightarrow T$ (the \textbf{multiplication}) such that the following diagrams commute:
\begin{equation*}
\xymatrix{T(X) \ar[dr]_{1_{T(X)}} \ar[r]^{T\eta_X} & T^2(X)
\ar@{}[dl]|(0.35){\circlearrowleft} \ar[d]_{\mu_X} \ar@{}[dr]|(0.35){\circlearrowleft}
& T(X) \ar[l]_{\eta_{T(X)}} \ar[dl]^{1_{T(X)}} \\ & T(X) & } 
\quad \quad 
\xymatrix{T^3(X)
\ar[d]_{T\mu_X} \ar[r]^{\mu_XT} \ar@{}[dr]|(0.5){\circlearrowleft} & T^2(X)
\ar[d]^{\mu_X} \\ T^2(X) \ar[r]_{\mu_X} & T(X)}
\end{equation*} 
that is, $\mu_X\circ T\eta_X=1_{T(X)}=\mu_X\circ \eta_{T(X)}$ and $\mu_X\circ
T\mu_X=\mu_X\circ \mu_XT$.

Every adjunction $F\dashv U$ with unit $\eta$ and counit $\epsilon$ gives rise to a unique monad $(UF,U\epsilon F,\eta)$.

Let $(T,\eta ,\mu )$ be a monad on a category $\mathscr{C}$.  An \textbf{algebra} for a monad $T$ consists of an object $A\in \mathscr{C}^0$ together with a morphism $TA \overset{\theta}{\rightarrow} A$ such that the following diagrams commute:
\begin{equation*}
\vcenter{\xymatrix{A \ar[r]^{\eta_A} \ar[dr]_{1_A} & TA \ar[d]^{\theta}
\ar@{}[dl]|(0.35){\circlearrowleft} \\ & A}} \quad  \quad 
\vcenter{\xymatrix{T^2A
\ar[d]_{T\theta} \ar[r]^{\mu_A} \ar@{}[dr]|(0.5){\circlearrowleft} & TA
\ar[d]^{\theta} \\ TA \ar[r]_{\theta} & A}}\quad 
\text{that is, $\theta \circ \eta_A =1_A$ and $\theta \circ T\theta = \theta \circ \mu_A$.}
\end{equation*}

\subsection{Penon Weak Higher Categories}\label{subsec: Penon}

Given an $\omega$-globular set $Q$, a reflexive globular $\omega$-magma $M$, with a morphism $\nu:Q\to M$ (as $\omega$-globular sets), is 
\textbf{free} over $Q$ if this universal factorization property holds: 
for every other morphism $\phi:Q\to \hat{M}$ (as $\omega$-globular sets) into another reflexive globular $\omega$-magma 
$\hat{M}$ there exists a unique morphism of reflexive globular $\omega$-magmas $\hat{\phi}:M\to \hat{M}$ such that $\phi = \hat{\phi}\circ \nu$.

Given an $\omega$-globular set $Q$, a strict globular $\omega$-category $C$, with a morphism $\nu:Q\to C$ (as $\omega$-globular sets), is 
\textbf{free} over $Q$ if this universal factorization property holds: 
for every other morphism $\phi:Q\to \hat{C}$ (as $\omega$-globular sets) into another strict globular $\omega$-category 
$\hat{C}$ there exists a unique morphism of strict globular $\omega$-categories $\hat{\phi}:C\to \hat{C}$ such that $\phi = \hat{\phi}\circ \nu$.

Note that free reflexive globular $\omega$-magmas (respectively, strict globular $\omega$-categories) over an \hbox{$\omega$-globular} set always exist (see~\cite{L2} and~\cite{P}) and, as for any definition via a universal factorization property, any two of them are canonically isomorphic.

Let $M$ be a reflexive globular $\omega$-magma, $C$ a strict globular $\omega$-category, and $\pi:M\to C$ a morphism of reflexive globular 
$\omega$-magmas.  A \textbf{Penon contraction for $\pi$} is a family of maps 
$[\cdot,\cdot]_q:\{(x,y)\in M^q\times M^q \ | \ s^{q-1}(x)=s^{q-1}(y), \ t^{q-1}(x)= t^{q-1}(y), \ \pi(x)=\pi(y)\}\to M^{q+1}$, for any 
$q\in \mathbb{N}$, satisfying the following three properties:
\begin{enumerate}
\item $s^q([x,y]_q)=x$ and $t^q([x,y]_q)=y$,
\item $x=y$ implies $[x,y]_q=\iota^q_M(x)=\iota^q_M(y)$,
\item $\pi([x,y]_q)=\iota^q_C(\pi(x))=\iota^q_C(\pi(y))$.
\end{enumerate}
Here below is a graphical depiction of Penon contractions: 
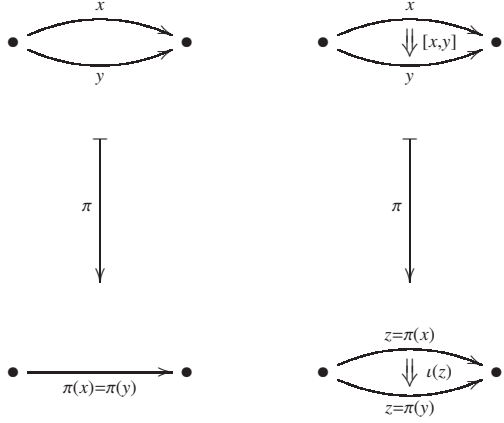
\begin{equation*}
\xymatrix{
\bullet \rrtwocell^x_y{\omit} & & \bullet 
\\ 
&  \ar@{|->}[dd]_{\pi} & 
\\ 
& &  
\\
& & 
\\
\bullet \ar[rr]_{\pi(x)=\pi(y)} & & \bullet
} \qquad \qquad  
\xymatrix{
\bullet \rrtwocell^x_y{\quad [x,y]} & & \bullet 
\\ 
&  \ar@{|->}[dd]_{\pi} & 
\\ 
& &  
\\
& & 
\\
\bullet \rrtwocell^{z=\pi(x)}_{z=\pi(y)}{\quad \iota(z)} & & \bullet 
}
\end{equation*}

We have a category of Penon contractions, where morphisms are defined as 
$$(M_1\overset{\pi_1}{\to}C_1, [\cdot,\cdot]^1)\xrightarrow{(\Phi,\phi)}(M_2\overset{\pi_2}{\to}C_2, [\cdot,\cdot]^2),$$ 
where $\Phi:M_1\to M_2$ is a morphism of reflexive globular $\omega$-magmas and $\phi:C_1\to C_2$ is an $\omega$-functor such that
$\pi_2\circ \phi=\Phi\circ \pi_1$ and $\Phi([x,y]^1_q)=[\Phi(x),\Phi(y)]^2_q$ for every $q\in \mathbb{N}$, $x,y$ in the domain of 
$[\cdot,\cdot]^1_q$.

There is a forgetful functor $U$ from the category of Penon contractions to the category $\mathscr{G}$ of \hbox{$\omega$-globular} sets associating to a contraction $(M\overset{\pi}{\to}C, [\cdot,\cdot])$ the underlying $\omega$-globular set of $M$.

J.Penon proved in~\cite{P} that $U$ admits a left adjoint functor $F\dashv U$ and gave the following: 

\begin{definition}
A \textbf{weak globular $\omega$-category} is an algebra for the monad $(UF,U\epsilon F,\eta)$.
\end{definition}

\section{Main Results}\label{sec: main} 

Our goal is a ``Penon's style'' treatment of self-dualities (involutions) for weak $\omega$-categories.  
Again we carefully distiguish the ``structural requirements'' in the definition of the unary operations of duality and the algebraic axioms necessary in the case of involutions. 
The material on self-dualities and strict involutive categories follows~\cite{BCLS,B} and is adapted/generalized to the case of $\omega$-quivers and 
$\omega$-magmas. 

\subsection{Self-Dual (Reflexive) Globular $\omega$-Quivers and $\omega$-Magmas}\label{subsec: sd quiver}

Let $\alpha \subseteq \mathbb{N}_0$. 
An \textbf{$\alpha$-contravariant morphism} $Q\xrightarrow{\phi}\hat{Q}$ of $\omega$-quivers or $\omega$-globular sets is a family of maps 
$\phi^n:Q^n\to\hat{Q}^n$ such that: 
\begin{itemize}
\item 
$\hat{s}^n \circ \phi^{n+1}=\phi^n \circ t^n, \quad \hat{t}^n \circ \phi^{n+1}=\phi^n \circ s^n,  \quad \forall n\in \alpha$; 
\item 
$\hat{s}^n \circ \phi^{n+1}=\phi^n \circ s^n, \quad \hat{t}^n \circ \phi^{n+1}=\phi^n \circ t^n,  \quad \forall n\notin \alpha$. 
\end{itemize}
For globular $\omega$-magmas, an $\alpha$-contravariant morphism must also satisfy: 
\begin{itemize}
\item 
$\phi^n(x\circ^n_p y)=\phi^n(y)\hat{\circ}^n_p\phi^n(x), \quad \forall n\in \alpha, \quad \forall (x,y)\in Q^n\times_pQ^n$, 
\item 
$\phi^n(x\circ^n_p y)=\phi^n(x)\hat{\circ}^n_p\phi^n(y), \quad \forall  n\notin \alpha, \quad \forall (x,y)\in Q^n\times_pQ^n$.
\end{itemize}
In the case of reflexive $\omega$-globular sets and reflexive globular $\omega$-magmas, $\alpha$-contravariant morphisms are furthermore required to satisfy: 
$\phi^n\circ \iota^{n-1}=\hat{\iota}^n\circ\phi^{n-1}$, for all $n\in \mathbb{N}$.  

A (reflexive) $\omega$-globular set $Q^0 \overset{s^0}{\underset{t^0}{\leftleftarrows}} Q^1 \overset{s^1}{\underset{t^1}{\leftleftarrows}} \cdots \overset{s^{n-1}}{\underset{t^{n-1}}{\leftleftarrows}} Q^n \overset{s^n}{\underset{t^n}{\leftleftarrows}} \cdots$ is \textbf{self-dual} if there exists a family of $\alpha$-contravariant morphisms 
$\ast^n_\alpha:Q^n\rightarrow Q^n$, for every $n\in \mathbb{N}_0$ and $\alpha \subseteq \mathbb{N}_0$, in detail: 
\begin{itemize}
\item $s^n(f^{\ast^{n+1}_\alpha})=t^n(f)^{*_\alpha^n}$ and $t^n(f^{\ast^{n+1}_\alpha})=s^n(f)^{*_\alpha^n}$
for every $n\in \alpha$ and $f\in Q^{n+1}$,
\item $s^n(f^{\ast^{n+1}_\alpha})=s^n(f)^{*_\alpha^n}$ and $t^n(f^{\ast^{n+1}_\alpha})=t^n(f)^{*_\alpha^n}$
for every $n\notin \alpha$ and $f\in Q^{n+1}$.
\end{itemize}
Similarly a \textbf{(reflexive) self-dual globular $\omega$-magma} is a (reflexive) globular $\omega$-magma whose underlying $\omega$-globular set is self-dual. Notice that in all these cases a self-duality is only an $\alpha$-contravariant morphism of $\omega$-globular sets, but it is not a morphism of reflexive $\omega$-globular sets or a morphism of (reflexive) $\omega$-magmas. 

The shape of $2$-cells related by self-dualities $*_{\varnothing},*_{\{0\}},*_{\{1\}},*_{\{0,1\}}$ are pictured here below: 
\\ 
$*_{\{1\}}: \xymatrix{A \rtwocell^f_g{x}& B}\quad \mapsto \quad \xymatrix{A \rrtwocell^f_g{^{\quad \quad \quad \ \ {x^{*_{\{1\}}}}}}& & B,}$
\quad  
$*_{\{0\}}: \xymatrix{A \rtwocell^f_g{x}& B}\quad \mapsto \quad \xymatrix{A &  
& \lltwocell_{f^{*_{\{0\}}}}^{g^{*_{\{0\}}}}{^{\quad \ x^{*_{\{0\}}}}}B,}$ \\ 
$*_{\{0,1\}}: \xymatrix{A \rtwocell^f_g{x}& B}\quad \mapsto \quad 
\xymatrix{A& & \lltwocell_{f^{*_{\{0\}}}}^{g^{*_\{0\}}}{\quad \quad \quad \ \ \ x^{*_{\{0,1\}}}}B,}$
\quad 
$*_{\varnothing}: \xymatrix{A \rtwocell^f_g{x}& B}\quad \mapsto \quad 
\xymatrix{A \rrtwocell^f_g{\quad x^{*_\varnothing}} & & B.}$

A \textbf{self-dual} morphism $Q\xrightarrow{\phi}\hat{Q}$ between self-dual $\omega$-quivers, $\omega$-globular sets or (reflexive) globular 
$\omega$-magmas is a morphism of the respective structures such that:  
$\phi^n(x^{*^n_\alpha})=\phi^n(x)^{\hat{*}^n_\alpha}$, for all $x\in Q^n$, for all $\alpha\subset\mathbb{N}$ and $n\in \mathbb{N}_0$. 

A \textbf{free (reflexive) self-dual globular $\omega$-magma} over an $\omega$-globular set $Q$, is a (reflexive) self-dual globular $\omega$-magma $M$ with a morphism of $\omega$-globular sets $\nu:Q\to M$ satisfying the following universal factorization property: for every morphism 
$\phi: Q\to \hat{M}$ (as $\omega$-globular sets) into a (reflexive) self-dual globular $\omega$-magma $\hat{M}$, there exists a unique morphism 
$\hat{\phi}:M\to\hat{M}$ of (reflexive) self-dual globular $\omega$-magmas such that $\phi=\hat{\phi}\circ\nu$. 
\textbf{Free (reflexive) self-dual $\omega$-globular sets} over an $\omega$-globular set, can be defined along the same lines. 

\begin{proposition}\label{prop: free-sd} 
Free self-dual reflexive globular $\omega$-magmas over an $\omega$-globular set $Q$ exist. 
\end{proposition}
\begin{proof}
The construction relies heavily on recursive arguments. 
Let $Q$ be an $\omega$-globular set. \\ 
Consider $\Gamma:=\{(\alpha_1,\dots,\alpha_m) \ | \ m\in \mathbb{N}, \ \forall k=1,\dots,m, \ \alpha_k\subset\mathbb{N}_0\}\cup\{\varnothing\}$ as a set of multi-indexes and, for $\gamma=(\alpha_1,\dots,\alpha_m)\in \Gamma$, the symmetric difference set 
$\bigtriangleup \gamma:=\alpha_1\bigtriangleup \cdots\bigtriangleup\alpha_m\subset \mathbb{N}_0$ .  

Let $M^0:=\{(x,\gamma) \ | \ x\in Q^0, \ \gamma\in \Gamma\}$. 
Define $M^1_\iota:=\{(z,\iota_1) \ | \ z\in M^0\}$ as a disjoint copy of $M^0$ and $s^0/t^0(z,\iota_1):=z$. 
Define $M^1[1]:=\{(z,\gamma) \ | \ z\in Q^1\cup M^1_\iota, \  \gamma\in\Gamma\}$ with 
$s^0[1]/t^0[1](z,\gamma):=(s^0/t^0(z),\gamma)$, if $0\notin\bigtriangleup \gamma$,  
$s^0[1]/t^0[1](z,\gamma):=(t^0/s^0(z),\gamma)$, if $0\in\bigtriangleup \gamma$. 

Suppose, by recursion, that we already defined $M^1[1],\dots,M^1[k-1]$ and $s^0/t^0$ on them, we define 
$M^1[k]:=\{((x,0,y),\gamma) \ | \ (x,y)\in M^1[i]\times_{M^0}M^1[j], \ i+j=k, \ \gamma\in \Gamma\}$ and we further set sources and targets as follows: 
\begin{gather*}
s^0[k]/t^0[k]((x,0,y),\gamma):=(s^0[k](y),\gamma)/(t^0[k](x),\gamma), \quad \text{if $0\notin \bigtriangleup\gamma$}, 
\\ 
s^0[k]/t^0[k]((x,0,y),\gamma):=(t^0[k](x),\gamma)/(s^0[k](x),\gamma), \quad \text{if $0\in \bigtriangleup\gamma$}. 
\end{gather*}  

Finally we define $M^1:=\cup_{k\in\mathbb{N}}M^1[k]$ and $s^0/t^0$ as ``union'' of the previous maps. 

Suppose, by further recursion, that we already defined $M^m$, for $m=0,\dots,n$, and all the maps $s^j/t^j$ on them. 
We define $M^{n+1}_\iota:=\{(z,\iota_{n+1}) \ | \ z\in M^{n}\}$, with $s^{n}/t^{n}(z,\iota_{n+1}):=z$. 

Similarly $M^{n+1}[1]:=\{(z,\gamma)\ | \ z\in Q^{n+1}\cup M^{n+1}_\iota, \gamma\in \Gamma\}$ with 
$s^{n}[1]/t^{n}[1](z,\gamma):=(s^{n}/t^{n}(z),\gamma)$, if $n\notin\bigtriangleup \gamma$, and 
$s^{n}[1]/t^{n}[1](z,\gamma):=(t^{n}/s^{n}(z),\gamma)$, if $n\in\bigtriangleup \gamma$. 
If we suppose, by recursion, already defined $M^{n+1}[1],\dots,M^{n+1}[k-1]$, and all source/target maps $s^{n}/t^{n}$ on them, we futher define $M^{n+1}[k]:=\{((x,p,y),\gamma) \ | \ p=0,\dots,n, \ (x,y)\in M^{n+1}[i]\times_{M^p} M^n[j], \ i+j=k, \ \gamma\in \Gamma\}$ and we set sources and targets as follows: 
\begin{align*}
&s^{n}[k]/t^{n}[k]((x,n,y),\gamma):=\begin{cases}
(s^{n}[k](y),\gamma)/(t^{n}(x)[k],\gamma),\  \text{if $n\notin \bigtriangleup\gamma$}, 
\\ 
(t^{n}[k](x),\gamma)/(s^{n}(x)[k],\gamma), \ \text{if $n\in \bigtriangleup\gamma$}, 
\end{cases} 
\\
&s^{n}[k]/t^{n}[k]((x,p,y),\gamma):=\begin{cases}
((s^{n}[k](x)/t^{n}[k](x), p, s^{n}[k](y)/t^{n}[k](y)),\gamma), 
\ \text{if $n\notin \bigtriangleup\gamma$, $p<n$}, 
\\ 
((t^{n}[k](x)/s^{n}[k](x), p, t^{n}[k](y)/s^{n}[k](y)),\gamma), 
\ \text{if $n\in \bigtriangleup\gamma$, $p<n$}. 
\end{cases}
\end{align*}
Finally we set $M^{n+1}:=\cup_{k\in \mathbb{N}}M^{n+1}[k]$, and $s^{n}/t^{n}$ the ``union'' of $s^{n}[k]/t^{n}[k]$.

The new $\omega$-quiver $M^0\leftleftarrows\cdots\leftleftarrows M^n\leftleftarrows \cdots$ is, by induction, an $\omega$-globular set; 
the nullary operations $\iota^n:M^{n-1}\to M^n$ are given by $z\mapsto ((z,\iota_n),\varnothing)$, for all $z\in M^{n-1}$;  
the unary operations $*^n_\alpha: M^n\to M^n$, for $\alpha\subset \mathbb{N}_0$, are given by $(z,\gamma)\mapsto (z,\gamma\oplus\{\alpha\})$, where we assume $(\alpha_1,\dots,\alpha_m)\oplus\{\alpha\}:=(\alpha_1,\dots.\alpha_m,\alpha)\in \Gamma$;  
the binary compositions $\circ^n_p:M^n\times_{M^p}M^n\to M^n$ are simply $(x,y)\mapsto (x,p,y)$.  
All the previous operations inductively satisfy the structural axioms for a self-dual reflexive globular $\omega$-magma. 

We only need to check the universal factorization property for the globular $\omega$-magma $M$ with the inclusion map $\nu:Q\to M$ given by 
$x\mapsto (x,\varnothing)\in M^n[1]\subset M^n$, for all $x\in Q^n$. 
For this purpose, let $\phi: Q\to\hat{M}$ a morphism of $\omega$-globular sets into another self-dual reflexive globular $\omega$-magma $\hat{M}$. The only possible choice of a map $\hat{\phi}:M\to\hat{M}$ such that $\phi=\hat{\phi}\circ\nu$, must necessarily satisfy 
$(x,\varnothing)\mapsto \phi(x)$ and, by recursion, using the fact that $\hat{\phi}$ is a morphism of self-dual reflexive globular $\omega$-magmas, we obtain, for all $n\in \mathbb{N}$, $((x,p,y),\gamma)\mapsto (\phi(x)\hat{\circ}^n_p\phi(y))^{\hat{*}_{\alpha_1}\cdots\hat{*}_{\alpha_m}}$, where $\gamma=(\alpha_1,\dots,\alpha_n)$. By induction this well-defined unique morphism $\hat{\phi}$ is a morphism of self-dual reflexive 
$\omega$-magmas such that $\phi=\hat{\phi}\circ\nu$ and this completes the proof. 
\end{proof}
Along similar lines, one can actually produce recursive construnctions of free (reflexive) self-dual $\omega$-globular sets and (reflexive) globular $\omega$-magmas over a given $\omega$-globular set. 

\subsection{Involutive Strict Globular $\omega$-Categories}\label{subsec: inv cat}

An \textbf{$\alpha$-contravariant functor} $C\xrightarrow{\phi}\hat{C}$ between strict globular $\omega$-categories is an $\alpha$-contravariant morphism of the undelying reflexive globular $\omega$-magmas. 

An \textbf{involutive strict globular $\omega$-category}\footnote{Here we are exactly following the definition put forward in~\cite{BCLS,B} for the case of $n$-categories.} is a strict globular $\omega$-category that is also a self-dual $\omega$-globular set with self-dualities $*_\alpha$, with 
$\alpha\subset \mathbb{N}_0$ that are $\alpha$-contravariant functors that further satisfy the following algebraic axioms: 
\begin{itemize}
\item
$(x^{*_\alpha})^{*_\alpha}=x, \quad \forall x\in C, \quad \forall \alpha\subset\mathbb{N}$, 
\item 
$(x^{*_\alpha})^{*_\beta}=(x^{*_\beta})^{*_\alpha}, \quad \forall x\in C, \quad \forall \alpha,\beta\subset\mathbb{N}$.
\end{itemize}
A $*$-functor between involutive strict globular $\omega$-categories is just a functor 
$C\xrightarrow{\phi}\hat{C}$ such that: $\phi(x^{*_\alpha})=\phi(x)^{\hat{*}_\alpha}$ for all $x\in C$ and for all $\alpha\subset \mathbb{N}$. 

A \textbf{free involutive strict globular $\omega$-category} over an $\omega$-globular set $Q$, is an involutive strict globular $\omega$-category $C$, with a morphism of $\omega$-globular sets $\nu:Q\to C$, satisfying the following universal factorization property: for every morphism 
$\phi: Q\to \hat{C}$ (as $\omega$-globular sets) into an involutive strict globular $\omega$-category $\hat{C}$, there exists a unique $*$-functor 
$\hat{\phi}:M\to\hat{M}$ such that $\phi=\hat{\phi}\circ\nu$. Unicity up to a unique isomorphism of involutive strict globular $\omega$-categories commuting with the inclusion morphisms is standard from the universal factorization. The existence can be obtained by a recursive construction, as in the previous case of a free self-dual globular $\omega$-magma, but we present here an alternative ``quotient'' argument starting from the already available free self-dual reflexive globular $\omega$-magmas over the $\omega$-globular set $Q$. 

Let $M:=M^0\leftleftarrows \cdots \leftleftarrows M^n\leftleftarrows \cdots$ be a self-dual reflexive globular $\omega$-magma. 
Consider its Cartesian product $M\times M:=(M^0\times M^0)\leftleftarrows \cdots \leftleftarrows (M^n\times M^n)\leftleftarrows \cdots$, where, for all $n\in\mathbb{N}_0$, the source and target maps $s^n_{M\times M}:=(s^n_M,s^n_M)$, $t^n_{M\times M}:=(t^n_M,t^n_M)$, as well as the 
structural nullary $\iota^n_{M\times M}:=(\iota^n_M,\iota^n_M)$, unary $(x,y)^{*_\alpha^{M\times M}}:=(x^{*_\alpha^M},y^{*_\alpha^M})$, 
and (when they exist) binary operations $(x_1,y_1)\circ^{(M\times M)^n}_p(x_2,y_2):=(x_1\circ^{M^n}_p y_1,x_2\circ^{M^n}_p y_2)$ are defined componentwise. 
In this way, the product $M\times M$ is another self-dual reflexive globular $\omega$-magma. 

A \textbf{congruence} in the self-dual reflexive globular $\omega$-magma $M$ is a self-dual reflexive globular $\omega$-magma 
$R$ such that, for all $n\in \mathbb{N}_0$, $R^n\subset M^n\times M^n$ is an equivalence relation in $M^n$ and the inclusions 
$\nu_n$ provide a morphism of self-dual reflexive globular $\omega$-magmas $\nu:R\to M\times M$. 

Under such conditions, we obtain a \textbf{quotient self-dual reflexive globular $\omega$-magma} $M/R$ with the quotient sets 
$M^n/R^n=:(M/R)^n$, $n\in \mathbb{N}_0$, with sources/targets given by $s^n_{M/R}([x]_{{n+1}}):=[s^n_M(x)]_n$, 
$t^n_{M/R}([x]_{{n+1}}):=[t^n_M(x)]_n$, compositions (whenever existing) defined as $[x]_n\circ^{(M/R)^n}_p[y]_n:=[x\circ^{M^n}_p y]_n$ and nullary operations $\iota^n_{M/R}([x]_n)=[\iota^n_M(x)]_{n+1}$. 
Furthermore, the quotient maps $\pi^n:M^n\to M^n/R^n$ onto the quotient sets give us a morphism $\pi:M\to M/R$ between self-dual reflexive globular $\omega$-magmas.\footnote{In a perfectly analogous way, one can introduce congruences and quotients for all the other ``intermediate'' structures between $\omega$-globular sets and self-dual reflexive globular $\omega$-magmas.} 
\begin{proposition}\label{prop: free-ic}
Free involutive strict globular $\omega$-categories over an $\omega$-globular set $Q$, exist. 
\end{proposition}
\begin{proof}
Let $Q\xrightarrow{\nu}M$ be the free strict self-dual reflexive globular $\omega$-magma over $Q$ as constructed in 
proposition~\ref{prop: free-sd}. 
In order to obtain from $M$ an involutive strict globular \hbox{$\omega$-category}, we must impose all the ``algebraic axioms'' (structural axioms for the globularity of $\omega$-quiver and for the ``domain/codomain'' of the nullary, unary and binary operations are already in place in $M$). 

We consider the congruence $R$ in $M$ ``generated'' by all the possible pairs of terms involved in the expression of the algebraic axioms: 
\begin{align*}
X:=&\{(x,(x^{*_\alpha})^{*_\alpha}) \ | \ x\in M, \ \alpha\subset \mathbb{N}_0\}\cup
\{((x^{*_\alpha})^{*_\beta}, (x^{*_\beta})^{*_\alpha}) \ | \ x\in M, \ \alpha,\beta\subset \mathbb{N}_0\}\cup 
\\ 
&\{((x\circ_py)^{*_\alpha}, (x^{*_\alpha})\circ_p (y^{*_\alpha})) \ | \ (x,y)\in M\times_{M^p}M, \ \mathbb{N}_0\ni p\notin\alpha\subset{N}_0\}\cup
\\ 
&\{((x\circ_py)^{*_\alpha}, (y^{*_\alpha})\circ_p (x^{*_\alpha})) \ | \ (x,y)\in M\times_{M^p}M, \ \mathbb{N}_0\ni p\in\alpha\subset{N}_0\}\cup 
\\
&\{(x\circ_p(y\circ_pz),(x\circ_py)\circ_p z) \ | \ (x,y,z)\in M\times_{M^p}M\times_{M^p}M, \ p\in\mathbb{N}_0\}\cup 
\\ 
&\begin{aligned}
\{((x\circ_py)\circ_q(z\circ_pw), (x\circ_qz)\circ_p(y\circ_qw)) \ | \ &(x,y),(z,w)\in M\times_{M^p}M, \ 
\\ 
&(x,z),(y,w)\in M\times_{M^q}M,\ p,q\in \mathbb{N}_0\} \cup
\end{aligned}
\\ 
&\{(\iota(x)\circ_q\iota(y),\iota(x\circ_qy)) \ | \ (x,y)\in M\times_{M^q}M, \ q\in \mathbb{N}_0 \}\cup
\\ 
&\{(\iota(x^{*_\alpha}),\iota(x)^{*_\alpha}) \ | \ x\in M, \ \alpha\subset\mathbb{N}_0\}\cup  
\\ 
&\{((\iota^{n-1}\circ\cdots\iota^p\circ t^p\circ\cdots\circ t^{n-1} (x))\circ^n_p x, x) \ | \ n,p \in \mathbb{N}_0,\  x\in M^n\}\cup  
\\ 
&\{(x, x\circ^n_p (\iota^{n-1}\circ\cdots\iota^p\circ s^p\circ\cdots\circ s^{n-1} (x))) \ | \ n,p \in \mathbb{N}_0,\  x\in M^n\}.
\end{align*}
this is by definition the smallest congruence in $M$ containing $X$. 
Such a congruence always exists and (since the arbitrary intersection of congruences is a congruence and $M\times M$ is always a congruence containing $X$) it coincides with the intersection of all congruences in $M$ containing $X\subset M\times M$. 
Taking now the quotient self-dual reflexive globular $\omega$-magma $M/R$, we note that, since $X\subset R$, all the algebraic axioms are already satisfied in $M/R$ and hence $M/R$ is already an involutive strict globular $\omega$-category. 

We only need to check that $Q\xrightarrow{\pi\circ \nu}M/R$ is a free involutive strict globular $\omega$-category via the universal factorization property. Let $\phi: Q\to \hat{C}$ be a morphism of $\omega$-globular sets into an involutive strict $\omega$-category $\hat{C}$. Since $M$ is a free self-dual reflexive globular $\omega$-magma over $Q$, there exists one and only one morphism of self-dual reflexive globular $\omega$-magmas 
$\overline{\phi}:M\to \hat{C}$ such that $\phi=\overline{\phi}\circ\nu$.  
Consider, for all $n\in \mathbb{N}_0$, $R_{\phi}^n:=\{(x,y)\in M^n\times M^n \ | \ \overline{\phi}_n(x)=\overline{\phi}_n(y)\}$. 
Since $\overline{\phi}$ is a morphism of self-dual reflexive globular $\omega$-magmas, $R_\phi$ becomes a congruence in $M$ and, thanks to the fact that $\hat{C}$ is already an involutive strict globular $\omega$-category, we have $X\subset R_\phi$ and hence $X/R_\phi$ is already an involutive strict globular $\omega$-category, and the assignment $\widetilde{\phi}:[x]_{R_\phi}\mapsto \overline{\phi}(x)$, for $x\in M$, is a well-defined 
$*$-functor $\widetilde{\phi}:M/R_\phi\to\hat{C}$ and it is the unique map such that $\widetilde{\phi}\circ\pi_\phi=\overline{\phi}$, where 
$\pi_\phi:M\to M/R_\phi$ denotes the quotient morphism. 

Since $R$ is the smallest congruence containing $X$, we have $R\subset R_\phi$ and hence there is a unique well-defined map $\theta:M/R\to M/R_\phi$ via the assignment $\theta: [x]_R\mapsto [x]_{R_\phi}$, for all $x\in M$, and $\theta$ is a $*$-functor of involutive strict globular 
$\omega$-categories and actually the unique map such that $\pi_\phi=\theta\circ\pi$. Combining the equations, we see that 
$\hat{\phi}:=\widetilde{\phi}\circ \theta:M/R\to\hat{C}$ is a $*$-functor and it is the unique morphism such that 
$\phi=\overline{\phi}\circ\nu=\widetilde{\phi}\circ \pi_\phi\circ\nu=\widetilde{\phi}\circ\theta\circ \pi\circ \nu=\hat{\phi}\circ(\pi\circ\nu)$.  
\end{proof}

\subsection{Involutive Weak Globular $\omega$-Categories}\label{subsec: w inv cat} 

Let $M$ be a self-dual reflexive globular $\omega$-magma, $C$ an involutive strict globular $\omega$-category,
and $\pi:M\to C$ a self-dual morphism of self-dual reflexive globular $\omega$-magmas.  
\\ 
Finally let $[\cdot,\cdot]_n$, $n\in \mathbb{N}$ be a usual Penon contraction for $\pi$, exactly as defined in section~\ref{subsec: Penon}. 

We have a category $\mathscr{Q}_\omega^*$ of ``self-dual'' Penon contractions, where morphisms are defined as 
$$(M_1\overset{\pi_1}{\to}C_1, [\cdot,\cdot]^1)\xrightarrow{(\Phi,\phi)}(M_2\overset{\pi_2}{\to}C_2, [\cdot,\cdot]^2),$$ 
where $\Phi:M_1\to M_2$ is a self-dual morphism of self-dual reflexive globular $\omega$-magmas and $\phi:C_1\to C_2$ is a $*$-functor of involutive strict $\omega$-categories, such that $\pi_2\circ \Phi=\phi\circ \pi_1$ and $\Phi([x,y]^1_q)=[\Phi(x),\Phi(y)]^2_q$ for every 
$q\in \mathbb{N}$, $x,y$ in the domain of $[\cdot,\cdot]^1_q$.

There is a forgetful functor $U^*$ from the category $\mathscr{Q}_\omega^*$ of ``self-dual'' Penon contractions to the category $\mathscr{G}$ of
$\omega$-globular sets, associating to a ``self-dual'' contraction $(M\overset{\pi}{\to}C, [\cdot,\cdot])$ the underlying $\omega$-globular set of $M$ (forgetting self-dualities, compositions and reflexive maps). 

A \textbf{free self-dual Penon contraction} over an $\omega$-globular set $Q$ is a self-dual Penon contraction 
$(M\overset{\pi}{\to}C, [\cdot,\cdot])$, with a morphism of $\omega$-globular sets $\nu: Q\to U^*((M\overset{\pi}{\to}C, [\cdot,\cdot]))$, such that the following universal factorization property holds: for any other morphism of $\omega$-globular sets 
$Q\xrightarrow{\phi} U^*(\hat{M}\overset{\hat{\pi}}{\to}\hat{C}, \widehat{[\cdot,\cdot]})$ into the undelying $\omega$-globular set of another self-dual Penon contraction $(\hat{M}\overset{\hat{\pi}}{\to}\hat{C}, \widehat{[\cdot,\cdot]})\in \mathscr{Q}_\omega^*$, there exists a unique morphism 
$(M\overset{\pi}{\to}C, [\cdot,\cdot])\xrightarrow{(\hat{\Phi},\hat{\phi})}(\hat{M}\overset{\hat{\pi}}{\to}\hat{C}, \widehat{[\cdot,\cdot]})$
in $\mathscr{Q}_\omega^*$ such that $U^*(\hat{\Phi},\hat{\phi})\circ \nu=\phi$. 
\begin{proposition}
Free self-dual Penon contractions over an $\omega$-globular set exist. 
\end{proposition}
\begin{proof}
The construction proceeds by recursion merging techniques from propositions~\ref{prop: free-sd} and~\ref{prop: free-ic}. 

Let $Q$ be an $\omega$-globular set. We construct $M^0=C^0=Q^0$ and $\pi^0:M^0\to C^0$ as the identity. 
Note that the domain of $[\cdot,\cdot]_0$ is empty (there is no contraction induced by $\pi^0$).  
Using the same notations as in the proof of propositions~\ref{prop: free-sd} and~\ref{prop: free-ic}, we define $M^1$, $C^1:=M^1/R^1$ and 
$\pi^1:M^1\to C^1$ as the quotient map by the congruence $R^1\subset M^1\times M^1$ generated by all the algebraic axioms $X^1$ between 
1-arrows of the free self-dual reflexive globular $\omega$-magma. 
Note that now the domain of $[\cdot,\cdot]_1$ concides with $X^1$. We define on $(x,y)\in X^1\subset M^1\times M^1$, $s^1(x,y):=x$ and 
$t^1(x,y):=y$. Next we set $M^2[1]:=\{(z,\gamma) \ | \ z\in Q^2\cup M^2_\iota\cup X^1, \gamma\in \Gamma\}$ (note the introduction of extra 
2-arrows coming from the contractions relative to the algebraic axioms in $X^1$) and we proceed exactly as in the proof of 
proposition~\ref{prop: free-sd} to recursively define $M^2[k]$, for all $k\in\mathbb{N}$ and get $M^2:=\cup_{k\in\mathbb{N}}M^2[k]$ as well as the source/target maps $s^2/t^2$. \\ 
We define now $C^2:=M^2/[R]^2$, where $[R]^2$ is the congruence generated by the algebraic axioms $[X]^2$  in $M^2$, $\pi^2:M^2\to C^2$ is the quotient map and the contraction $[\cdot,\cdot]_1:X_1\to M^2$ is the inclusion 
$X^1\subset M^2[1]\subset M^2$. Note that the set $[X]^2$ now contains also the axioms for the contractions: 
$\{(s^1([x,y]_1),x) \ | \ (x,y)\in X^1\}\cup \{(t^1([x,y]_1),y)\ | \ (x,y)\in X^1\}\cup\{([x,x]_1,(x,\iota_1)) \ | \ x\in M^1\}$.
 
If we suppose, by recursion, that we already defined $\pi^n:M^n\to C^n$, $[\cdot,\cdot]_{n-1}:[X]^{n-1}\to M^n$  as above, we can consider 
$X^n\subset M^n\times M^n$ as the set of algebraic axioms between $n$-arrows; define 
$M^{n+1}[1]:=\{(z,\gamma) \ | \ z\in Q^{n+1}\cup M^{n+1}_\iota\cup [X]^n, \gamma\in \Gamma\}$ and 
$M^{n+1}:=\cup_{k\in \mathbb{N}}M^{n+1}[k]$; the contraction $[\cdot,\cdot]_n:[X]^n\to M^{n+1}$ always as inclusion; the congruence generated by the algebraic axioms $[X]^{n+1}$ between $(n+1)$-arrows $[R]^{n+1}\subset M^{n+1}\times M^{n+1}$ and finally obtain 
$\pi^{n+1}$ as the quotient map onto $C^{n+1}:=M^{n+1}/[R]^{n+1}$, completing the recursive step of the definition. 

The nullary, unary and binary operations on the new $\omega$-quiver $M$ are defined as in proposition~\ref{prop: free-sd} (there are only the extra arrows coming from $X$ to be considered). Inductively $M$ turns out to be a self-dual reflexive globular $\omega$-magma, the quotient $C=M/R$ by the congruence $R$ is a strict involutive $\omega$-category, since $X\subset R$, and $\pi:M\to C$ is a morphism of self-dual reflexive globular 
$\omega$-magmas. The union of all the maps $[\cdot,\cdot]_n:[X]^n\to M^{n+1}$ is a contraction. 
The inclusion $\nu$ of $Q$ into $U^*((M\xrightarrow{\pi}C), [\cdot,\cdot])$ is simply the map $x\mapsto (x,\varnothing)$ as before. 

We only need to show the universal factorization property. For this purpose, let $(\Phi,\phi)$ be a morphism in $\mathscr{Q}_\omega^*$ into a new self-dual contraction $(\hat{M}\xrightarrow{\hat{\pi}},\hat{C})\in \mathscr{Q}_\omega^*$. If $(\hat{\Phi},\hat{\phi})$ is a morphism in 
$\mathscr{Q}_\omega^*$ such that $U^*(\Phi,\phi)=U^*(\hat{\Phi},\hat{\phi})\circ \nu$, we necessarily have $\hat{\Phi}(x,\varnothing)=\Phi(x)$, for all $x\in Q$. 
Since $\hat{\Phi}$ is a morphism of self-dual reflexive globular $\omega$-magmas, the definition of $\hat{\Phi}$ is uniquely given by 
$\hat{\Phi}^n(z,\gamma)=\Phi^n(z)^{\hat{*}_{\alpha_1}\cdots \hat{*}_{\alpha_m}}$, if $z\in Q^n$ and $\gamma:=(\alpha_1,\dots,\alpha_m)$; 
$\hat{\Phi}^n(z,\gamma)=\hat{\iota}(z)^{\hat{*}_{\alpha_1}\cdots \hat{*}_{\alpha_m}}$, if $z\in M^n_\iota$, for $n\in \mathbb{N}_0$; 
$\hat{\Phi}^n([x,y]_{n-1},\gamma)= \widehat{[\Phi^n(x),\Phi^n(y)]}_{n-1}^{\hat{*}_{\alpha_1}\cdots \hat{*}_{\alpha_m}}$, 
if $[x,y]_{n-1}\in [X]^{n-1}$, for $n\in\mathbb{N}$. 
An inductive argument shows that this unique map $\hat{\Phi}:M\to\hat{M}$ is actually a morphism of self-dual reflexive globular $\omega$-magmas.
The map $\hat{\pi}\circ\hat{\Phi}:M\to\hat{C}$ induces the congruence $R_{\hat{\pi}\circ\hat{\Phi}}$ in $M$ and since $\hat{\Phi}$ preserves the contractions, we have $[X]\subset R_{\hat{\pi}\circ\hat{\Phi}}$ and hence, for the congruence $R$ in M generated by $[X]$, 
$R\subset R_ {\hat{\pi}\circ\hat{\Phi}}$. It follows that there exists a unique induced $*$-functor $\hat{\phi}:C\to\hat{C}$ such that 
$\hat{\pi}\circ\hat{\Phi}=\hat{\phi}\circ \pi$ and hence $(\hat{\Phi},\hat{\phi})$ is the unique morphism in $\mathscr{Q}_\omega^*$ such that 
$\phi=U^*(\hat{\Phi},\hat{\phi})$ and we completed the proof of the universal factorization property. 
\end{proof}

\begin{theorem}
The forgetful functor $U^*:\mathscr{Q}_\omega^*\to \mathscr{G}$ admits a left adjoint $F^*\dashv U^*$. 
\end{theorem}
\begin{proof}
We define $F^*$ on the objects of $\mathscr{G}$ as the map associating to an $\omega$-globular set $Q$ the specific free self-dual Penon contraction $F^*(Q)$ constructed in the previous proposition and  let $\eta_Q:Q\to U^*(F^*(Q))$ denote the ``inclusion'' morphism in the definition of the free self-dual Penon contraction. 

If $\gamma:Q_1\to Q_2$ is a morphism in $\mathscr{G}$, we have that $\eta_2\circ\gamma:Q_1\to U^*(F^*(Q_2))$ is a morphism in $\mathscr{G}$ and hence, by the universal factorization property for free self-dual Penon contractions, there exists a unique morphism 
$F^*_\gamma: F^*(Q_1)\to F^*(Q_2)$ in $\mathscr{Q}_\omega^*$ such that $U^*(F^*_\gamma)\circ \eta_1=\eta_2\circ \gamma$. 

The map $\gamma\mapsto F^*_\gamma$ is functorial from $\mathscr{G}\to\mathscr{Q}_\omega^*$ and, by standard arguments about adjunction, 
$\eta: Q\mapsto \nu_Q$ is the unit of an adjunction $F^*\dashv U^*$.  
\end{proof}
Finally we can provide our main definition: 
\begin{definition}
A \textbf{Penon weak involutive globular $\omega$-category} is defined as an algebra for the monad $(U^*F^*,U^*\epsilon^* F^*,\eta^*)$.
\end{definition}

\subsection{Examples}\label{subsec: ex} 

We just mention here, without entering into a detailed discussion, some of the most immediate examples of involutive weak categories. 

\begin{example}
Every strict involutive globular $\omega$-category is a very particular trivial case of weak involutive globular $\omega$-category. In particular strict globular $\omega$-groupoids.
\end{example}

\begin{example}
Weak $\omega$-groupoids are just special cases of weak involutive $\omega$-categories with involutions given by (suitable composition of) the inverses. In particular the most elementary and well-known examples fitting our definition of weak involutive $\omega$-category are the 
\textbf{fundamental $\omega$-groupoids} $\Pi_\omega(X)$ of topological spaces $X$ (see~\cite[page~xiv-xv]{L2}). 

Let $X$ be a topological spaces, $\Pi_\omega(X)^0:=X$, $\Pi_\omega(X)^1:=C([0,1];X)$ is the set of continuous paths in $X$, $\Pi_\omega(X)^2$ is the set of homotopies of paths with fixed endopoints, \dots, $\Pi_\omega(X)^n$ is the set of homotopies between $(n-1)$-homotopies, etc. 
Compositions of homotopies are defined in the usual way and involutions consist of the inverse homotopies. 
\end{example}

\begin{example}
Truncations, at the level of $n$-arrows, of involutive strict $\omega$-categories are involutive strict $n$-categories and, in the other direction, involutive strict $n$-categories become involutive strict $\omega$-categories, just taking identities as the only morphisms for all $m>n$. The situation for weak categories is more involved: an involutive weak $n$-category can be defined as an algebra for a similar monad associated to the adjunction 
$\xymatrix{\mathscr{Q}^*_n \rtwocell^{F_n}_{U_n} {'\perp}& \mathscr{G}_n}$ between the forgetful functor $U_n$ and its left adjoint functor $F_n$ between the category of $n$-globular sets $\mathscr{G}_n$ and the Penon self-dual contraction category $\mathscr{Q}^*_n$.  
\end{example}

\begin{example}
Globular $\omega$-quivers (and more generally the ``globular'' propagators of globular $\omega$-quivers discussed in~\cite{BJ}) are examples of weak involutive globular $\omega$-categories. 
\end{example}

Of particular motivation for us is the following example of ``higher Morita categories''. 
\begin{example}
Let $\mathscr{M}^0$ be a family of involutive monoids $A,B,C,\dots$ and $\mathscr{M}^1$ the family of the bimodules ${}_AM_B$, with 
$A,B\in \mathscr{M}^0$. Composition $\circ^1_0$ of bimodules is given by the Rieffel tensor product ${}_AM_B \otimes_B {}_BN_C$ and involution $*^1_0$ of bimodules is provided by the Rieffel dual ${}_B\overline{M}_A$ where $\overline{M}:=\{\overline{x} \ | \ x\in M\}$ is just a (specific) disjoint copy of $M$ and the bimodule actions are $b\cdot \overline{x}\cdot a:=\overline{a^*xb^*}$, for all $a\in A$, $b\in B$ and $x\in M$. 
Similarly starting from a class $\mathscr{M}^0$ of strict involutive 1-categories, the family $\mathscr{M}^1$ of ``bimodules'' between them is a weak involutive $1$-category. Introducing a suitable notion of ``bimodule'' between strict involutive $n$-categories, we obtain a weak involutive 
$n$-category. If $\mathscr{M}^0$ is a family of strict $\omega$-categories, the family $\mathscr{M}^1$ of ``bimodules'' between them is a weak involutive $\omega$-category. 
\end{example}

\medskip 

\emph{Notes and Acknowledgments:}
This research was financially supported by the Research Professional Development Project under the Science Achievement Scholarship of Thailand (SAST) and Thammasat University, Faculty of Science and Technology, by Thammasat University research grant n.~2/15/2556: ``Categorical Non-commutative Geometry''. 

\medskip 

The second author thanks Starbucks Coffee at the $1^{\text{st}}$ floor of Emporium Tower and Jasmine Tower, in Sukhumvit, where he spent a significant part of the time dedicated to this research project.

\end{document}